\documentclass[12pt]{amsart}
\usepackage{amsfonts}
\usepackage{amscd}
\usepackage{amssymb}
\usepackage{graphics}

\usepackage{amsmath}
\usepackage{graphicx,psfrag}
\usepackage{subcaption,mathtools}

\usepackage{hyperref}
\hypersetup{colorlinks,citecolor=blue}

\newcommand{\la}{\langle}
\newcommand{\ra}{\rangle}
\newtheorem{theorem}{Theorem}

\newtheorem{corollary}[theorem]{Corollary}

\newtheorem{definition}[theorem]{Definition}

\newtheorem{lemma}[theorem]{Lemma}

\newtheorem{proposition}[theorem]{Proposition}
\newtheorem{remark}[theorem]{Remark}

\newtheorem*{Theorem2}{Theorem 2}

\newcommand{\Cl}{\mathrm{Cl}}

\newcommand{\zz}{\mathbb{Z}[\frac{1}{2}]}

\begin{document}
\author{Gili Golan Polak}\thanks{The research was supported by ISF grant 2322/19.}

\title{Thompson's group $F$ is almost $\frac{3}{2}$-generated}

\begin{abstract}
	Recall that a group $G$ is said to be $\frac{3}{2}$-generated if every non-trivial element of $G$ belongs to a generating pair of $G$. Thompson's group $V$ was proved to be $\frac{3}{2}$-generated by Donoven and Harper in 2019. It was the first example of an infinite finitely presented non-cyclic $\frac{3}{2}$-generated group. 
Recently, Bleak, Harper and Skipper proved that Thompson's group $T$ is also $\frac{3}{2}$-generated. 
In this paper, we prove that Thompson's group $F$ is ``almost'' $\frac{3}{2}$-generated in the sense that every element of $F$ whose image in the abelianization forms part of a generating pair of $\mathbb{Z}^2$ is part of a generating pair of $F$. We also prove that for every  non-trivial element $f\in F$ there is an element $g\in F$ such that the subgroup $\la f,g\ra$ contains the derived subgroup of $F$. Moreover, if 
$f$ does not belong to the derived subgroup of $F$, then there is an element $g\in F$ such that $\la f,g\ra$ has finite index in $F$. 
\end{abstract}

\maketitle

\section{Introduction}

A group $G$ is said to be $\frac{3}{2}$-generated if every non-trivial element of $G$ is part of a generating pair of $G$. In 2000, settling a problem of Steinberg from 1962 \cite{S}, Guralnick and Kantor proved that all finite simple groups are $\frac{3}{2}$-generated \cite{GK}. 
In 2008, Breuer, Guralnick and Kantor \cite{BGK} observed that if a group $G$ is $\frac{3}{2}$-generated then every proper quotient of it must be cyclic. They conjectured that for finite groups this is also a sufficient condition.
The conjecture was proved in 2021 by Burness, Guralnick and Harper \cite{BGH}.

Note that for an infinite group $G$ every proper quotient being cyclic is not a sufficient condition for the group  being $\frac{3}{2}$-generated. Indeed, the infinite alternating group $A_{\infty}$ is simple but not finitely generated. Moreover, there are finitely generated infinite simple groups which are not $2$-generated (see \cite{Guba})
and in particular, are not $\frac{3}{2}$-generated. Recently, Cox \cite{Cox} constructed an example of an infinite $2$-generated group $G$ such that every proper quotient of $G$ is cyclic and yet $G$ is not $\frac{3}{2}$-generated. 

Obvious examples of infinite $\frac{3}{2}$-generated groups are the Tarski monsters constructed by Olshanskii \cite{O}. Recall that Tarski monsters are infinite finitely generated non-cyclic groups where every proper subgroup is cyclic\footnote{There are two types of Tarski monsters. One where every proper subgroup is infinite cyclic and one where every proper subgroup is cyclic of order $p$ for some fixed prime $p$.}. In particular, if $T$ is a Tarski monster, then $T$ is generated by any pair of non-commuting elements of $T$. Since the center of $T$ is trivial, every non-trivial element of $T$ is part of a generating pair of $T$. 

In 2019, Donoven and Harper gave the first examples of infinite non-cyclic $\frac{3}{2}$-generated groups, other than Tarski monsters. Indeed, they proved that Thompson's group $V$ is $\frac{3}{2}$-generated. 
More generally, they proved that all  Higman–Thompson groups $V_n$ (see \cite{Higman}) and all Brin–Thompson groups $nV$ (see \cite{Brin}) are $\frac{3}{2}$-generated.
In 2020, Cox constructed two more examples of infinite $\frac{3}{2}$-generated groups with some special properties (see \cite{Cox}). Quite recently (in 2022), Bleak, Harper and Skipper proved that Thompson's group $T$ is also $\frac{3}{2}$-generated \cite{BHS}. 

In this paper, we study Thompson's group $F$. Recall that Thompson's group $F$ is the group of all piecewise-linear homeomorphisms  of the interval $[0,1]$ with finitely many breakpoints where all breakpoints are dyadic fractions (i.e., numbers from $\mathbb{Z}[\frac{1}{2}]\cap (0,1)$) and all slopes are integer powers of $2$. Thompson's group $F$ is $2$-generated. The derived subgroup of $F$ is infinite and simple and can be characterized as the subgroup of $F$ of all functions $f$ with slope $1$ both at $0^+$ and at $1^-$ (see \cite{CFP}). 
The abelianization $F/[F,F]$ is isomorphic to $\mathbb{Z}^2$. The standard abelianization map $\pi_{ab}\colon F\to \mathbb{Z}^2$ maps every function $f\in F$ to $(\log_2(f'(0^+)),\log_2(f'(1^-)))$ (see \cite{CFP}). Since the abelianization of $F$ is $\mathbb{Z}^2$, Thompson's group $F$ cannot be $\frac{3}{2}$-generated. However, we prove that it is ``almost'' $\frac{3}{2}$-generated in the sense that the following theorem holds. 

\begin{theorem}\label{thm:almost3/2}
	Every element of $F$ whose image in the abelianization $\mathbb{Z}^2$ is part of a generating pair of $\mathbb{Z}^2$ is part of a generating pair of $F$. 
\end{theorem}

In fact, we have the following more general theorem. 

\begin{theorem}\label{thm:main intro}
	Let $(a,b),(c,d)\in\mathbb{Z}^2$ be such that $\{a,c\}\neq\{0\}$ and $\{b,d\}\neq\{0\}$. Let $f\in F$ be a non-trivial element such that $\pi_{ab}(f)=(a,b)$, then there is an element $g\in F$ such that $\pi_{ab}(g)=(c,d)$ and such that $$\la f,g\ra =\pi_{ab}^{-1}(\la (a,b),(c,d)\ra).$$
\end{theorem}

Recall that a subgroup $H$ of $F$ is normal if and only if it contains the derived subgroup of $F$ \cite{CFP} (in particular, all finite index subgroups of $F$ are normal subgroups of $F$).
It follows that the normal subgroups of $F$ are exactly the subgroups $\pi_{ab}^{-1}(\la(a,b),(c,d)\ra)$ for $(a,b),(c,d)\in\mathbb{Z}^2$. Note that if $a=c=0$ or $b=d=0$, then $\pi_{ab}^{-1}(\la(a,b),(c,d)\ra)$ is not finitely generated (see Remark \ref{rem:not finitely generated} below). Theorem \ref{thm:main intro} implies that all other normal subgroups of $F$ are $2$-generated. Hence, we have the following.

\begin{corollary}
	Every finitely generated normal subgroup of $F$ is $2$-generated. %
\end{corollary}

In particular, every finite index subgroup of $F$ is $2$-generated. 
Recall that in \cite{BW} (see also \cite{BCR}), the finite index subgroups of $F$ that are isomorphic to $F$ were characterized. Let $p,q\in\mathbb{N}$. We denote by $F_{p,q}$ the subgroup $\pi_{ab}^{-1}(p\mathbb{Z}\times q\mathbb{Z})$ of $F$.
Then for every $p,q\in\mathbb{N}$ the subgroup $F_{p,q}$ is isomorphic to $F$ and these are the only finite index subgroups of $F$ that are isomorphic to $F$ \cite{BW}. In particular, the finite index  subgroups $F_{p,q}$ of $F$ are known to be $2$-generated. 
   It was conjectured in \cite[Conjecture 12.6]{G16} that all finite index subgroups of $F$ are $2$-generated.

Note  that for every non-trivial $(a,b)\in\mathbb{Z}^2$, there exists $(c,d)\in\mathbb{Z}^2$ such that $\la(a,b),(c,d)\ra$ is a finite index subgroup of $\mathbb{Z}^2$ of the form $p\mathbb{Z}\times q\mathbb{Z}$ for some $p,q\in\mathbb{N}$ (see Lemma \ref{lem:linear algebra} below). Hence, Theorem \ref{thm:main intro} implies the following.

\begin{corollary}
	Let $f\in F$ be an element whose image in the abelianization of $F$ is non-trivial. Then there is an element $g\in F$ such that the subgroup $\la f,g\ra$ is isomorphic to $F$ and has finite index in $F$. 
\end{corollary}

Note that Theorem \ref{thm:almost3/2}  shows that in some sense it is ``easy'' to generate Thompson's group $F$.  
Several other results demonstrate (in different ways) the abundance of generating pairs of Thompson's group $F$. Recall that in \cite{G-RG}, we prove that in the two  natural probabilistic models studied in \cite{CERT}, a random pair of elements of $F$ generates $F$ with positive probability. 
In \cite{GGJ}, Gelander, Juschenko and the author proved that Thomspon's group $F$ is invariably generated by $3$ elements (i.e., there are $3$ elements $f_1,f_2,f_3\in F$ such that regardless of how each one of them is conjugated, together they generate $F$.) Using results from \cite{G22}, one can show that in fact, there is a  pair of elements in $F$ which invariably generates $F$.  


\section{Preliminaries}\label{s:FT}

\subsection{F as a group of homeomorphisms}

Recall that Thompson group $F$ is the group of all piecewise linear homeomorphisms of the interval $[0,1]$ with finitely many breakpoints where all breakpoints are  dyadic fractions and all slopes are integer powers of $2$.  
The group $F$ is generated by two functions $x_0$ and $x_1$ defined as follows \cite{CFP}.

\[
x_0(t) =
\begin{cases}
	2t &  \hbox{ if }  0\le t\le \frac{1}{4} \\
	t+\frac14       & \hbox{ if } \frac14\le t\le \frac12 \\
	\frac{t}{2}+\frac12       & \hbox{ if } \frac12\le t\le 1
\end{cases} 	\qquad	
x_1(t) =
\begin{cases}
	t &  \hbox{ if } 0\le t\le \frac12 \\
	2t-\frac12       & \hbox{ if } \frac12\le t\le \frac{5}{8} \\
	t+\frac18       & \hbox{ if } \frac{5}{8}\le t\le \frac34 \\
	\frac{t}{2}+\frac12       & \hbox{ if } \frac34\le t\le 1 	
\end{cases}
\]

The composition in $F$ is from left to right.

Every element of $F$ is completely determined by how it acts on the set $\zz$. Every number in $(0,1)$ can be described as $.s$ where $s$ is an infinite word in $\{0,1\}$. For each element $g\in F$ there exists a finite collection of pairs of (finite) words $(u_i,v_i)$ in the alphabet $\{0,1\}$ such that every infinite word in $\{0,1\}$ starts with exactly one of the $u_i$'s. The action of $F$ on a number $.s$ is the following: if $s$ starts with $u_i$, we replace $u_i$ by $v_i$. For example, $x_0$ and $x_1$  are the following functions:

\[
x_0(t) =
\begin{cases}
	.0\alpha &  \hbox{ if }  t=.00\alpha \\
	.10\alpha       & \hbox{ if } t=.01\alpha\\
	.11\alpha       & \hbox{ if } t=.1\alpha\
\end{cases} 	\qquad	
x_1(t) =
\begin{cases}
	.0\alpha &  \hbox{ if } t=.0\alpha\\
	.10\alpha  &   \hbox{ if } t=.100\alpha\\
	.110\alpha            &  \hbox{ if } t=.101\alpha\\
	.111\alpha                      & \hbox{ if } t=.11\alpha\
\end{cases}
\]
where $\alpha$ is any infinite binary word.

The group $F$ has the following finite presentation \cite{CFP}.
$$F=\la x_0,x_1\mid [x_0x_1^{-1},x_1^{x_0}]=1,[x_0x_1^{-1},x_1^{x_0^2}]=1\ra,$$ where $a^b$ denotes $b^{-1} ab$. 

\subsection{Elements of F as pairs of binary trees} \label{sec:tree}

Often, it is more convenient to describe elements of $F$ using pairs of finite binary trees (see \cite{CFP} for a detailed exposition). The considered binary trees are rooted \emph{full} binary trees; that is, each vertex is either a leaf or has two outgoing edges: a left edge and a right edge. A  \emph{branch} in a binary tree is a simple path from the root to a leaf. If every left edge in the tree is labeled ``0'' and every right edge is labeled ``1'', then a branch in $T$ has a natural binary label. We rarely distinguish between a branch and its label. 

Let $(T_+,T_-)$ be a pair of finite binary trees with the same number of leaves. The pair $(T_+,T_-)$ is called a \emph{tree-diagram}. Let $u_1,\dots,u_n$ be the (labels of) branches in $T_+$, listed from left to right. Let $v_1,\dots,v_n$ be the (labels of) branches in $T_-$, listed from left to right. We say that the tree-diagram $(T_+,T_-)$ has the \emph{pair of branches} $u_i\rightarrow v_i$ for $i=1,\dots,n$. The tree-diagram $(T_+,T_-)$ \emph{represents} the function $g\in F$ which takes binary fraction $.u_i\alpha$ to $.v_i\alpha$ for every $i$ and every infinite binary word $\alpha$. We also say that the element $g$ takes the branch $u_i$ to the branch $v_i$.
For a finite binary word $u$, we denote by $[u]$  the dyadic interval $[.u,.u1^{\mathbb{N}}]$. If $u\rightarrow v$ is a pair of branches of $(T_+,T_-)$, then $g$ maps the interval $[u]$ linearly onto $[v]$. 

A \emph{caret} is a binary tree composed of a root with two children. If $(T_+,T_-)$ is a tree-diagram and one attaches a caret to the $i^{th}$ leaf of $T_+$ and the $i^{th}$ leaf of $T_-$ then the resulting tree diagram is \emph{equivalent} to $(T_+,T_-)$ and represents the same function in $F$. The opposite operation is that of \emph{reducing} common carets. A tree diagram $(T_+,T_-)$ is called \emph{reduced} if it has no common carets; i.e, if there is no $i$ for which the  $i$ and ${i+1}$ leaves of both $T_+$ and $T_-$ have a common father. The reduced tree-diagrams of the generators $x_0$ and $x_1$ of $F$ are depicted in Figure \ref{fig:x0x1}.

\begin{figure}[ht]
	\centering
	\begin{subfigure}{.55\textwidth}
		\centering
		\includegraphics[width=.55\linewidth]{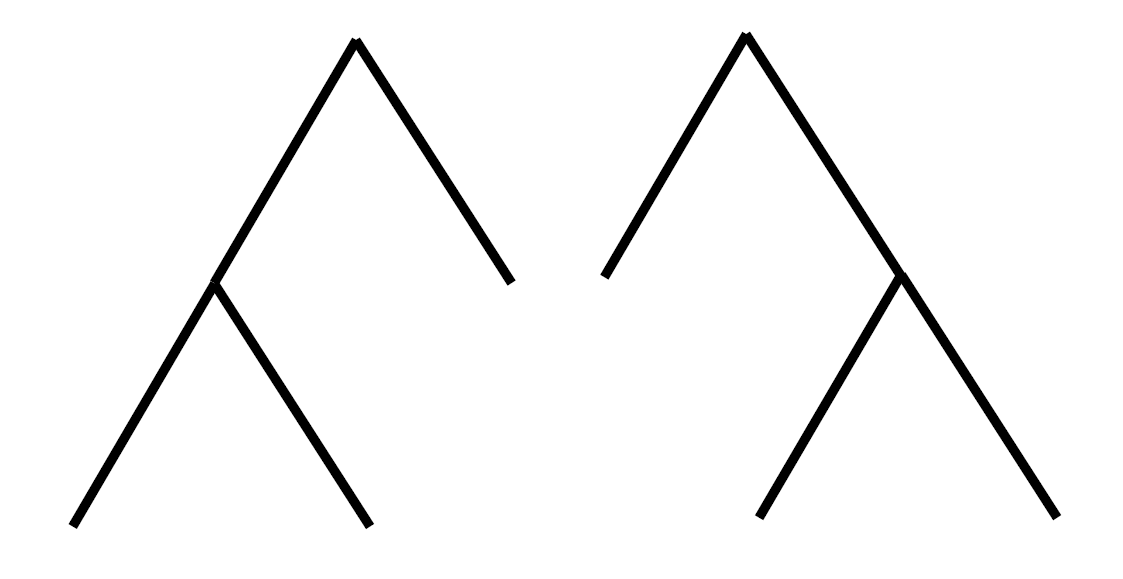}
		\caption{}
		\label{fig:x0}
	\end{subfigure}%
	\begin{subfigure}{.55\textwidth}
		\centering
		\includegraphics[width=.55\linewidth]{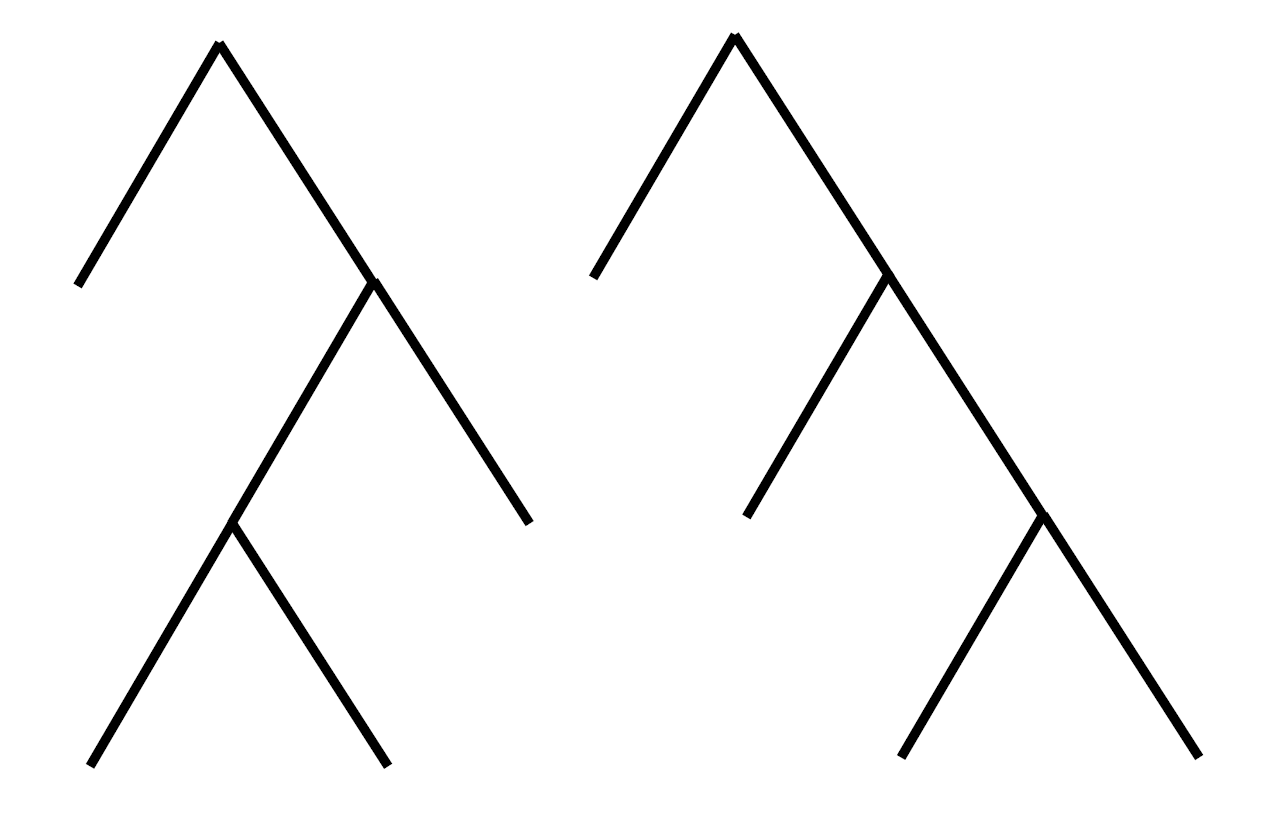}
		\caption{}
		\label{fig:x1}
	\end{subfigure}
	\caption{(A) The reduced tree-diagram of $x_0$. (B) The reduced tree-diagram of $x_1$. In both figures, $T_+$ is on the left and $T_-$ is on the right.}
	\label{fig:x0x1}
\end{figure}

 When we say that a function $f\in F$ has a pair of branches $u_i\rightarrow v_i$, the meaning is that some tree-diagram representing $f$ has this pair of branches. In other words, this is equivalent to saying that $f$ maps the dyadic interval $[u_i]$ linearly onto $[v_i]$.
Clearly, if $u\rightarrow v$ is a pair of branches of $f$, then for any finite binary word $w$, $uw\rightarrow vw$ is also a pair of branches of $f$. Similarly, if $f$ has the pair of branches $u\rightarrow v$ and $g$ has the pair of branches $v\rightarrow w$ then $fg$ has the pair of branches $u\rightarrow w$. 

Let $\mathcal B$ be the set of all finite binary words and let $u,v\in\mathcal B$. We say that $v$ is a \emph{descendant} of $u$ if $u$ is a strict prefix of $v$. We say that $u$ and $v$ are \emph{incomparable} if $u$ is not a prefix of $v$ and $v$ is not a prefix of $u$. Note that if $u$ and $v$ are incomparable then the interiors of $[u]$ and $[v]$ have empty intersection. In that case, we will write  $[u]<[v]$ if for every $x$ in the interior of $[u]$ and every $y$ in the interior of $[v]$ we have $x<y$. The following lemma will be useful. 

\begin{lemma}\label{lem:obvious}
	Let $f$ be a non-trivial element of $F$. Then there exist $u,v,w\in\mathcal B$ such that $[u]<[v]<[w]$ and such that $f$ or $f^{-1}$ has the pairs of branches $u\to v$ and $v\to w$. 
\end{lemma}

\begin{proof}
	Let $\alpha$ be the maximal number in $[0,1]$ such that $f$ fixes the interval $[0,\alpha]$ pointwise. Since $f$ is non-trivial, $\alpha<1$. 
	Note that $f'(\alpha^+)\neq 1$. If $f'(\alpha^+)>1$, let $g=f$, otherwise, let $g=f^{-1}$, so that $g'(\alpha^+)>1$.  Since $\alpha$ is a breakpoint of $f$, it must be dyadic. Hence, there exists a finite binary word $s$ 
	such that $\alpha=.s$. 
	 Since $g'(\alpha^+)> 1$, there exist $n>m$ in $\mathbb{N}$ such that $g$ has the pair of branches $.s0^n\to .s0^m$ (see \cite[Lemma 2.6]{G16}). 
	Now, let $u\equiv s0^{n+(n-m)}1$, $v\equiv s0^n1$ and $w\equiv s0^m1$. Then $[u]<[v]<[w]$ and the function $g$ has the pairs of branches $u\to v$ and $v\to w$, as necessary.  
\end{proof}

\subsection{The derived subgroup and the abelianization of $F$}\label{sec:derived}

Recall that the derived subgroup of $F$ is an infinite simple group which can be characterized as the subgroup of $F$ of all functions $f$ with slope $1$ both at $0^+$ and at $1^-$ (see \cite{CFP}). 
That is, a function $f\in F$ belongs to $[F,F]$ if and only if the reduced (equiv. any) tree-diagram of $f$ has pairs of branches of the form $0^m\rightarrow 0^m$ and $1^n\rightarrow 1^n$ for some $m,n\in\mathbb{N}$.

As noted above, a subgroup $H$ of $F$ is normal if and only if it contains the derived subgroup of $F$. In particular, every finite index subgroup of $F$ is a normal subgroup of $F$ \cite{CFP}. 
Recall that the abelianization of $F$ is isomorphic to $\mathbb{Z}^2$ and that the standard abelianization map $\pi_{ab}\colon F\to\mathbb{Z}^2$ maps an element $f\in F$ to $(\log_2(f'(0^+)),\log_2(f'(1^-)))$.
We make the following observation. 

\begin{remark}\label{rem:not finitely generated}
	Let $(a,b),(c,d)\in\mathbb{Z}^2$ be such that $a=c=0$ or $b=d=0$. Then the subgroup $\pi_{ab}^{-1}(\la (a,b),(c,d)\ra)\leq F$ is not finitely generated.
\end{remark}

\begin{proof}
	We prove it in the case where $a=c=0$. Note that in that case, $\la (a,b),(c,d)\ra=\la (0,g)\ra$, where $g=\gcd(b,d)$. Let $H=\pi_{ab}^{-1}(\la (0,g)\ra)$. We  claim that $H$ is not finitely generated. Assume by contradiction that $H$ is generated by $f_1,\dots,f_n$, for $n\in\mathbb{N}$. Then for each $i=1,\dots,n$ we have $f_i'(0^+)=1$. Hence, for each $i$, $f_i$ fixes a right neighborhood of $0$. Let $\alpha\in (0,1)$ be a small enough dyadic fraction such that for each $i=1,\dots,n$, $f_i$ fixes the interval $[0,\alpha]$ pointwise. Since $H$ is generated by $f_1,\dots,f_n$, it follows that $H$ fixes the interval $[0,\alpha]$ pointwise. But it is easy to construct an element $f\in F$ such that  $\pi_{ab}(f)=(0,g)$ and such that $f$ does not fix $\alpha$, a contradiction. 
\end{proof}

Recall that for every pair of elements $(a,b),(c,d)\in\mathbb{Z}^2$ the index of the subgroup $\la (a,b),(c,d)\ra$ of $\mathbb{Z}^2$ is the absolute value of the determinant $ad-bc$ (where if $ad-bc=0$, the index is infinite). The following simple lemma was referred to in the introduction. As we were unable to find a reference for it, we provide a proof.

\begin{lemma}\label{lem:linear algebra}
	Let $(a,b)\in\mathbb{Z}^2$ be non-trivial. Then there exists $(c,d)\in\mathbb{Z}^2$ and $p,q\in\mathbb{N}$ such that 
	$\la (a,b),(c,d)\ra=p\mathbb{Z}\times q\mathbb{Z}$ and such that $pq=\gcd(a,b)$. 
\end{lemma}

\begin{proof}
	If $a=0$ then for $(c,d)=(1,0)$ and $p=1$, $q=|b|$ we have the result. Indeed, $$\la (a,b),(c,d)\ra=\la (0,b),(1,0)\ra=\mathbb{Z}\times |b|\mathbb{Z}.$$ Similarly, if $b=0$, we are done. Hence, we can assume that $a,b\neq 0$. 
	Let $g=\gcd(a,b)$ and let $a',b'\in\mathbb{Z}$ be such that $a=ga'$ and $b=gb'$. Clearly, $\gcd(a',b')=1$. Let $g=p_1^{n_1}\cdots p_m^{n_m}$ be the prime factorization of $g$. Since $\gcd(a',b')=1$, each of the primes $p_i$ for $i=1,\dots,m$ divides at most one of the numbers $a'$ or $b'$. Let $q$ be the product of all $p_i^{n_i}$, $i=1,\dots,m$ such that $p_i$ divides $a'$ (if there are no such $p_i$'s, we let $q=1$). Let $p\in\mathbb{N}$ be such that $g=pq$ (in other words, $p$ is the product of all $p_i^{n_i}$ such that $p_i$ does not divide $a'$). By construction, $q$ and $b'$ are co-prime, $p$ and $a'$ are co-prime and $p$ and $q$ are co-prime. It follows that $qa'$ and $pb'$ are co-prime. Hence, by the Euclidean algorithm there exist $m,n\in\mathbb{Z}$ such that $m(qa')-n(pb')=1$. Let $(c,d)=(np,mq)$. 
	Then the index of the subgroup $\la (a,b),(c,d)\ra$ in $\mathbb{Z}^2$ is 
	$$|ad-bc|=|(ga')(mq)-(gb')(np)|=g|mqa'-npb'|=g=pq.$$
	On the other hand, since $p$ divides both $a=ga'$ and $c=np$, we have that $p|\gcd(a,c)$.  Similarly, $q$ divides $\gcd(b,d)$. Hence, 
	$$\la (a,b),(c,d)\ra\leq \gcd(a,c)\mathbb{Z}\times \gcd(b,d)\mathbb{Z}\leq p\mathbb{Z}\times q\mathbb{Z}.$$
	Since $p\mathbb{Z}\times q\mathbb{Z}$ is an over-group of $\la (a,b),(c,d)\ra$ and they both have index $pq$ in $\mathbb{Z}^2$, they must coincide. Hence, 
	$$\la (a,b),(c,d)\ra= p\mathbb{Z}\times q\mathbb{Z},$$
	as necessary. 
\end{proof}

\subsection{Generating normal subgroups of $F$}

Let $H$ be a subgroup of $F$. A function $f\in F$ is said to be a \emph{piecewise-$H$} function if there is a finite subdivision of the interval $[0,1]$ such that on each interval in the subdivision, $f$ coincides with some function in $H$. Note that since all breakpoints of elements in $F$ are dyadic fractions, a function $f\in F$ is a piecewise-$H$ function if and only if there is a  dyadic subdivision of the interval $[0,1]$ into finitely many pieces such that on each dyadic interval in the subdivision, $f$ coincides with some function in $H$. 

Following \cite{GS,G16}, we define the \emph{closure} of a subgroup $H$ of $F$, denoted $\Cl(H)$, to be the subgroup of $F$ of all piecewise-$H$ functions. A subgroup $H$ of $F$ is \emph{closed} if $H=\Cl(H)$. In \cite{G16}, we gave the following characterization of subgroups of $F$ which contain the derived subgroup of $F$ (as well as an algorithm for determining if a finitely generated subgroup of $F$ contains the derived subgroup of $F$).

\begin{theorem}\cite[Theorem 7.10]{G16}\label{gen}
	Let $H$ be a subgroup of $F$. Then $H$ contains the derived subgroup of $F$ (equiv. $H$ is a normal subgroup of $F$) if and only if the following conditions hold. 
	\begin{enumerate}
		\item[(1)] $\Cl(H)$ contains the derived subgroup of $F$. 
		\item[(2)] There is an element $h\in H$ and a dyadic fraction $\alpha\in (0,1)$ such that $h$ fixes $\alpha$, $h'(\alpha^-)=1$ and 
		$h'(\alpha^+)=2$.
	\end{enumerate}
\end{theorem}

Note that if a subgroup $H$ of $F$ contains the derived subgroup of $F$, then the image of $H$ in the abelianization of $F$ completely determines the subgroup $H$. 

\section{Subgroups $H$ of $F$ whose closure is a normal subgroup of $F$}

In the next section, we apply Theorem \ref{gen} to prove that a given subset of $F$ generates a normal subgroup $H$ of $F$. 
To do so, we will have to prove in particular that $H$ satisfies Condition (1) of Theorem \ref{gen}, i.e., that the closure of $H$ contains the derived subgroup of $F$. To that end, with each subgroup of $F$, we associate an equivalence relation on the set of finite binary words $\mathcal B$.

\begin{definition}
	Let $H$ be a subgroup of $F$.
 The \emph{equivalence relation induced by $H$} on the set of finite binary words $\mathcal B$, denoted $\sim_H$, is defined as follows. For every pair of finite binary words $u,v\in\mathcal B$ we have $u\sim_H v$ if and only if there is an element in $H$ with the pair of branches $u\to v$.
\end{definition}

Note that if $u\sim_Hv$ then for every finite binary word $w$, we have $uw\sim_H vw$. 
Moreover, if $H$ is closed, we have the following. 

\begin{lemma}\label{coherent}
	Let $H$ be a closed subgroup of $F$. Then for every pair of finite binary words $u,v\in\mathcal B$ we have $u\sim_H v$ if and only if $u0\sim_H v0$ and $u1\sim_H v1$. 
\end{lemma}

\begin{proof}
	Assume that $u0\sim_H v0$ and $u1\sim_Hv1$. We claim that there is an element in $H$ with the pair of branches $u\to v$. 
	Since $u0\sim_H v0$, there is an element $h_1\in H$ with the pair of branches $u0\to v0$.
	Similarly, since $u1\sim_H v1$, there is an element  $h_2\in H$ with the pair of branches $u1\to v1$. Let $\alpha=.u01^\mathbb{N}=.u1$ and $\beta=.v01^\mathbb{N}=.v1$. Then $h_1(\alpha)=\beta$ and $h_2(\alpha)=\beta$. Let $h_3\in F$ be the following function
	\[h_3(x)=\begin{cases}
		h_1(x) & x\in [0,\alpha]\\
		h_2(x) & x\in [\alpha,1]
	\end{cases}
	\]
	(Since $\alpha$ is dyadic, $h_3$ is indeed in $F$.) By construction, $h_3$ is a piecewise-$H$ function. 
	 Hence, since $H$ is closed, the element $h_3\in H$. But the element $h_3$ has the pairs of branches $u0\to v0$ and $u1\to v1$. Hence, $h_3$ maps the interval $[u0]$ linearly onto $[v0]$ and the interval $[u1]$ linearly onto $[v1]$. It follows that $h_3$ maps the interval $[u]$ linearly onto the interval $[v]$. Hence, it has the pair of branches $u\to v$. Therefore, $u\sim_H v$. 
\end{proof}


\begin{remark}
	An equivalence relation $\sim$ on the set of finite binary words $\mathcal B$ such that for every $u,v\in\mathcal B$ we have $u\sim v$ if and only if $u0\sim v0$ and $u1\sim v1$ is called a \emph{coherent} equivalence relation in \cite{BBQS}. Coherent equivalence relations were used in \cite{BBQS} to study maximal subgroups of Thompson's group $V$. 
\end{remark}

The following corollary follows from Lemma \ref{coherent} by induction on $k$. 

\begin{corollary}\label{cor:coherent}
	Let $H$ be a closed subgroup of $F$. Let $u,v\in\mathcal B$ and let $k\in\mathbb{N}$. Assume that for every finite binary word $w$ of length $k$ we have $uw\sim_H vw$. Then $u\sim_H v$. 
\end{corollary}

Let $\mathcal B'$ be the set of all finite binary words which contain both digits ``0'' and ``1''. In other words, $\mathcal B'=\mathcal B\setminus\{\emptyset, 0^n,1^n\mid n\in\mathbb{N}\}$. 

\begin{lemma}\label{lem:all inner words identified}
	Let $H$ be a closed subgroup of $F$. Assume that for every pair of finite binary words $u,v\in\mathcal B'$ we have $u\sim_H v$. Then $H$ contains the derived subgroup of $F$. 
\end{lemma}

\begin{proof}
	 Let $f\in [F,F]$ . Then the reduced tree-diagram of $f$ consists of pairs of branches 
	\[
	f :
	\begin{cases}
		0^m & \rightarrow 0^m\\
		u_i  & \rightarrow v_i \mbox{ for } i=1,\dots,k \\
		1^n & \rightarrow 1^n\\
	\end{cases}
	\]
	where $k,m,n\in\mathbb{N}$ and where for each $i=1,\dots,k$, the binary words $u_i$ and $v_i$ contain both digits $``0"$ and $``1"$ and as such belong to $\mathcal B'$. By assumption, for each $i=1,\dots,k$ we have $u_i\sim_H v_i$ and as such there is an element $h_i\in H$ with the pair of branches $u_i\rightarrow v_i$. Then $h_i$ coincides with $f$ on the interval $[u_i]$. We note also that $f$ coincides with the identity function {\bfseries{1}} $\in H$ on $[0^m]$ and on $[1^n]$. Since $[0^m],[u_1],\dots,[u_k],[1^n]$ is a subdivision of the interval $[0,1]$ and on each of these intervals $f$ coincides with a function in $H$, $f$ is a piecewise-$H$ function. Since $H$ is closed, $f\in H$. 
\end{proof}

\begin{lemma}\label{lem:identified of length k}
	Let $H$ be a subgroup of $F$. Assume that there is $k\in\mathbb{N}$ such that for every $u,v\in\mathcal B'$ and every $s\in\mathcal B$ such that the length $|s|=k$ we have $us\sim_H vs$. Then $\Cl(H)$ contains the derived subgroup of $F$.  
\end{lemma}

\begin{proof}
	Since $H$ is contained in $\Cl(H)$, the equivalence relation $\sim_H$ is contained in $\sim_{\Cl(H)}$. By Lemma \ref{lem:all inner words identified}, to prove that $\Cl(H)$ contains the derived subgroup of $F$ it suffices to prove that for every $u,v\in\mathcal B'$ we have $u\sim_{\Cl(H)} v$. Let $u,v\in \mathcal B'$. By assumption, there exists $k\in\mathbb{N}$ such that for every finite binary word $s$ of length $k$ we have $us\sim_{\Cl(H)} vs$. But then, since $\Cl(H)$ is closed, by Corollary \ref{cor:coherent}, we have $u\sim_{\Cl(H)} v$, as necessary. 
\end{proof}

The next lemma is used in the next section to show that for certain subgroups $H$ of $F$, the closure of $H$ contains the derived subgroup of $F$. 
	
\begin{lemma}\label{lem:tree}
	Let $H$ be a subgroup of $F$ and let $T$ be a non-empty finite binary tree with branches $u_1,\dots,u_n$. Let $w\in B$ and assume that the following assertions hold. 
	\begin{enumerate}	
		\item $w\sim_H w0\sim_H w1$. 
		\item For every $i=2,\dots,n-1$ we have $u_i\sim_H w$. 
		\item for every $i\geq 0$, we have $u_10^i1\sim_H w$.
		\item for every $i\geq 0$, we have $u_n1^i0\sim_H w$.
	\end{enumerate}
	Then $\Cl(H)$ contains the derived subgroup of $F$. 
\end{lemma}

\begin{proof}
	First, note that since $w\sim_H w0\sim_H w1$, it follows by induction that for every finite binary word $q$, we have $wq\sim_H w$. In other words, every descendant $v$ of $w$ is $\sim_H$-equivalent to $w$. It follows, that for every $p\in \mathcal B$ such that $p\sim_H w$, every descendant of $p$ is also $\sim_H$-equivalent to $w$.  Now, let $k$ be the maximum length of a branch of $T$ and let $u,v\in\mathcal B'$. 
	By Lemma \ref{lem:identified of length k}, to prove that $\Cl(H)$ contains the derived subgroup of $F$ it suffices to prove that for every $s\in \mathcal B$ of length $k$ we have $us\sim_H vs$. Let $s\in \mathcal B$ of length $k$. We claim that $us\sim_H w$. 
	Indeed, since $|us|>k$, there is  $i\in\{1,\dots,n\}$ such that the branch $u_i$ of $T$ is a strict prefix of $us$. If $i\in\{2,\dots,n\}$, then since $u_i\sim_H w$ and $us$ is a descendant of $u_i$, we have that $us\sim_H w$ as necessary. Otherwise, either $u_1$ or $u_n$ is a strict prefix of $us$. We consider the case where $u_1$ is a strict prefix of $us$, the other case being similar. Since $u_1$ is a string of zeros and $u$ contains the digit $1$, the word $u_1$ must be a strict prefix of $u$. In fact, there exists $j\geq 0$ such that $u_10^j1$ is a prefix of $u$. Since $u_10^j1\sim_H w$ by assumption and since  $us$ is a descendant of $u_10^j1$, we get that $us\sim_H w$, as necessary. In a similar way, one can show that $vs\sim_H w$. Hence, $us\sim_H vs$, as required. 
\end{proof}

\section{Proof of the main theorems}

To prove Theorem \ref{thm:main intro}, we will need the following proposition. 

\begin{proposition}\label{prop:main}
	Let $f$ be a non-trivial element of $F$ and let $c,d\in\mathbb{Z}\setminus\{0\}$. Then the following assertions hold. 
	\begin{enumerate}
		\item[(1)] There is an element $g\in F$ such that $\pi_{ab}(g)=(c,d)$ and such that $\la f,g\ra$ contains the derived subgroup of $F$. 
		\item[(2)] If $f$ has non-trivial slope at $1^-$ then there is an element $g\in F$ such that $\pi_{ab}(g)=(c,0)$ and such that $\la f,g\ra$ contains the derived subgroup of $F$. 	
		\item[(3)] If $f$ has non-trivial slope at $0^+$ then there is an element $g\in F$ such that $\pi_{ab}(g)=(0,d)$ and such that $\la f,g\ra$ contains the derived subgroup of $F$. 
		\item[(4)] If $f$ has non-trivial slope both at $0^+$ and at $1^-$ then there is an element $g\in F$ such that $\pi_{ab}(g)=(0,0)$ (i.e., such that $g\in [F,F]$) and such that $\la f,g\ra$ contains the derived subgroup of $F$.
	\end{enumerate}
\end{proposition}

\begin{proof}
	We prove parts (1) and (2). The proofs of parts (3) and (4) are similar.

	(1) Since $f$ is non-trivial, by Lemma \ref{lem:obvious}, there exist finite binary words $u,v$ and $w$ such that $[u]<[v]<[w]$ and such that $f$ or $f^{-1}$ has the pairs of branches $u\to v$ and $v\to w$. It is easy to see that $u,v$ and $w$ must belong to $\mathcal B'$. Note that if (1) holds for $(c,d)$ it also holds for $(-c,-d)$. Hence, we can assume without loss of generality that $c>0$.  
	We will also assume that $d>0$, and we will explain below how to modify the proof for the case where $c>0$ and $d<0$. 
	
	Let $T$ be a finite binary tree such that $u,v0,v1,w0,w10$ and $w11$ are branches of $T$ and such that $T$ has at least three branches after the branch $w11$. Let $(S_+,S_-)$ be the reduced tree-diagram of $x_1$ (see Figure \ref{fig:x0x1}(B)).
	Let $T_1$ be the minimal binary tree with branch $0^{c+1}$ and let $T_2$ be the minimal binary tree with branch $1^c$. Let $C_1$ and $C_2$ be trees which consist of a single caret each. Let $T_3$ be the minimal binary tree with branch $1^{d+1}$ and $T_4$ be the minimal binary tree with branch $0^{d}$.
	
	Let $u_1,\dots,u_n$ be the branches of $T$.	 We will use two copies of the tree $T$ to construct a new tree-diagram in $F$. 
	Let $R_+$ be the tree obtained from the first copy of $T$ by performing the following operations:
	\begin{enumerate}
		\item[(1a)] Attaching the tree $T_1$ to the end of the branch $u_1$; 
		\item[(2a)] Attaching the tree $S_1$ to the end of the branch $w10$; 
		\item[(3a)] Attaching the tree $T_3$ to the end of the branch $u_n$.
	\end{enumerate}

	Let $R_-$ be the tree obtained from the second copy of $T$ by performing the following operations: 
	\begin{enumerate}
		\item[(1b)] Attaching the tree $T_2$ to the end of the branch $u_1$.
		\item[(2b)] Attaching the caret $C_1$ to the end of the branch $w0$.
		\item[(3b)] Attaching the tree $S_2$ to the end of the branch $w10$.
		\item[(4b)] Attaching the caret $C_2$ to the end of the branch $w11$.
		\item[(5b)] Attaching the tree $T_4$ to the end of the branch $u_n$.
	\end{enumerate}
	
	Note that in the construction of $R_+$, we added $(c+1)+3+(d+1)$ carets to the tree $T$. 
	Similarly, in the construction of $R_-$, $c+1+3+1+d$ carets were added to the tree $T$. Hence $R_+$ and $R_-$ have the same number of carets. Let $g$ be the element represented by the tree-diagram $(R_+,R_-)$. We claim that the assertion in Proposition \ref{prop:main}(1) holds for $g$. 
	
	Recall that $w0,w10,w11$ are (necessarily consecutive) branches of $T$. In addition, $w0$ cannot be one of the first $4$ branches of $T$, since $u_1,u,v0$ and $v1$ precede it. By assumption, there are at least three branches in $T$ to the right of the branch $w11$. Hence, there exists $5\leq k\leq n-5$ such that $u_k\equiv w0$, $u_{k+1}\equiv w10$ and $u_{k+2}\equiv w11$. By construction, the tree-diagram $(R_+,R_-)$ has the following sets of pairs of branches (we recommend to the reader verifying it to read separately the list of branches of $R_+$ (the branches on the left hand-side of $\mbox{(A),(B) and (C)}$) and the list of branches of $R_-$ (the branches on the right hand-side of $\mbox{(A),(B) and (C)}$)).
	\begin{align*}
	\mbox{(A)}& \begin{cases}
	u_10^{c+1}\to u_10\\
	u_10^{c+1-i}1\to u_11^i0, & \mbox{ for } 1\leq i\leq c-1\\
	u_101\to u_11^c\\
	u_11\to u_2\\	
	{u_i}\to u_{i+1} & \mbox{for } 2\leq i\leq k-2\\
	u_{k-1}\to w00\\
	u_k\equiv w0\to w01\\
\end{cases}\\
\mbox{(B)}& \begin{cases}
	w100 \to    w100& \\
	w10100\to		w1010& \\
	w10101\to		w10110& \\
	w1011\to		w10111& \\
\end{cases}
\end{align*}
\begin{align*}
\mbox{(C)}& \begin{cases}
	w11\to w110& \\
	u_{k+3}	\to w111\\
	u_{i} \to u_{i-1} & \mbox{for } k+4\leq i\leq n-1\\
	u_n0\to u_{n-1}\\
	u_n10\to u_n0^d\\
	u_n1^i0\to u_10^{d+1-i}1   & \mbox{ for } 2\leq i\leq d\\
	u_n1^{d+1}\to u_n 1\\
\end{cases}
\end{align*}

Note that the pairs of branches in $\mbox{(B)}$ are the pairs of branches of $(R_+,R_-)$ which result from the attachment of the trees $S_1$ and $S_2$ to the branch $w10$ of both copies of $T$ (see Figure \ref{fig:x0x1}(B) for the branches of $S_1$ and $S_2$). 
 The preceding pairs of branches of $(R_+,R_-)$ are all in $\mbox{(A)}$. In particular, $\mbox{(A)}$ contains all the pairs of branches of $(R_+,R_-)$ which result from the attachment of the tree $T_1$ to  the first copy of $T$ and the trees $T_2$ and $C_1$ to the second copy of $T$. Similarly, $\mbox{(C)}$ contains all the pairs of branches of $(R_+,R_-)$ which result from the attachment of the tree $T_3$ to  the first copy of $T$ and the trees $T_4$ and $C_2$ to the second copy of $T$.

Now, the first and last pairs of branches of $(R_+,R_-)$ show that $\pi_{ab}(g)=(c,d)$ as required. Let $H=\la f,g\ra$. It suffices to prove that $H$ contains the derived subgroup of $F$. First, note that the pairs of branches $w100\to w100$ and $w10100\to w1010$ in $\mbox{(B)}$ show that $g$ fixes the dyadic fraction $\alpha=.w101=.w1001^{\mathbb{N}}$ and that $g'(\alpha^-)=1$ and $g'(\alpha^+)=2$. Hence, $H$ satisfies condition (2) of Theorem \ref{gen}. Therefore, to prove that $H$ contains $[F,F]$ it suffices to prove that $\Cl(H)$ contains $[F,F]$. For that, we will make use of Lemma \ref{lem:tree}. Let us denote the equivalence relation $\sim_H$ by $\sim$. To prove that $\Cl(H)$ contains the derived subgroup of $F$ it suffices to prove that for the tree $T$ (with branches $u_1,\dots,u_n$) and the word $w$, conditions (1)-(4) of Lemma \ref{lem:tree} hold.

First, note that 
 since $f$ or $f^{-1}$ has the pairs of branches $u\to v$ and $v\to w$, we have $u\sim v\sim w$. 
Now, let us consider the pairs of branches of $g$. The pairs of branches $u_i\to u_{i+1}$, $2\leq i\leq k-2$ 
imply that $u_2\sim u_3\sim\cdots\sim u_{k-1}$. Recall that $u,v0$ and $v1$ are branches of $T$ (distinct from $u_1$) which precede $u_k\equiv w0$. 
As such, $u,v0,v1\in\{u_2,\dots,u_{k-1}\}$, so that $u\sim v0\sim v1$. Since $u\sim v$, we get that $v\sim v0\sim v1$. It follows that every descendant of $v$ is $\sim$-equivalent to $v$. Since $w\sim v$, every descendant of $w$ is $\sim$-equivalent to $w$ (and as such, Condition (1) of Lemma \ref{lem:tree} holds).
  In particular, since $u_k\equiv w0, u_{k+1}\equiv w10$ and $u_{k+2}\equiv w11$, we have $u_k\sim u_{k+1}\sim u_{k+2}\sim w$. Since $u_2\sim\cdots\sim u_{k-1}\sim u\sim w$, we have $u_2\sim\cdots\sim u_{k-1}\sim u_k\sim u_{k+1}\sim  u_{k+2}\sim w$.

Now, the pair of branches $u_{k+3}\to w111$ in $\mbox{(C)}$ implies that $u_{k+3}\sim w111\sim w$. The pairs of branches $u_i\to u_{i-1}$ for $k+4\leq i\leq n-1$ imply that $u_{k+3}\sim u_{k+4}\sim\cdots\sim u_{n-1}$. Since $u_{k+3}\sim w$, we have that $u_2\sim \cdots u_{k+2}\sim u_{k+3}\sim\cdots\sim u_{n-1}\sim w$. That is, for every $2\leq i\leq n-1$ we have $u_i\sim w$. Hence, Condition (2) of Lemma \ref{lem:tree} holds. 

Next, we show that Condition (3) of Lemma \ref{lem:tree} holds. That is, we claim that for every $i\geq 0$, we have $u_10^i1\sim w$. 
Indeed, for $i=0$ it holds since the pair of branches  $u_11\to u_2$ of the element $g$ implies that $u_11\sim u_2\sim w$. Moreover, since $u_11\sim  w$, every descendant of $u_11$ is  $\sim$-equivalent to $w$. 
 Hence, the pair of branches $u_101\to u_11^c$ of $g$ implies that $u_101\sim u_11^c\sim w$, so the claim holds for $i=1$ as well. 
Now, the pairs of branches $u_10^{c+1-j}1\to u_11^j0$ for $1\leq j\leq c-1$ of $g$ imply that for all  $1\leq j\leq c-1$ we have $u_10^{c+1-j}1\sim w$. Letting $i=c+1-j$, we get that for all $i=2,\dots,c$ we have $u_10^i1\sim w$. Hence, the claim holds for every $i=0,\dots,c$. We prove by induction that the claim also holds for every $i>c$. Indeed, let $i\geq c+1$. Then the pair of branches $u_10^{c+1}\to u_10$ of $g$ implies that 
$$u_10^i\equiv u_10^{c+1}0^{i-(c+1)}\sim
u_100^{i-(c+1)}\equiv u_10^{i-c}.$$
That implies that $u_10^i1\sim u_10^{i-c}1\sim w$, by induction. Hence, for all $i\geq 0$, we have $u_10^i1\sim w$ and Condition (3) of Lemma \ref{lem:tree} holds. 

Note that in the proof of Condition (3) we have made use of branches in $\mbox{(A)}$ (more specifically, of the branches written in the first $4$ rows in $\mbox{(A)}$). In an almost identical manner, using branches of $\mbox{(C)}$ (more specifically, the branches written in the last $4$ rows of $\mbox{(C)}$), one can show that Condition (4) from Lemma \ref{lem:tree} holds. It follows that $\Cl(H)$ contains the derived subgroup of $F$. Hence, by Theorem \ref{gen}, $H=\la f,g\ra$ contains the derived subgroup of $F$, as required.

It remains to note that one can modify the proof for the case where $c>0$ and $d<0$ as follows. First, wherever $d$ appears in the above proof, we replace it by $|d|$. Second, in the construction of $R_+$ and $R_-$, instead of performing operation $(3a)$ on the first copy of $T$ (used in the construction of $R_+$) and operations $(4b)$ and $(5b)$ on the second copy of $T$ (used in the construction of $R_-$), we perform operation $(3a)$ on the second copy of $T$ and operations $(4b)$ and $(5b)$ on the first copy of $T$. The result is a tree-diagram $(R'_+,R'_-)$, whose pairs of branches coincide with the pairs of branches of $(R_+,R_-)$, other than the ones listed in $\mbox{(C)}$: for every pair of branches $p\to q$ in $\mbox{(C)}$, the tree-diagram $(R'_+,R'_-)$ has the ``opposite'' pair of branches $q\to p$. Note that this change, does not affect the proof  that the subgroup generated by $f$ and the element represented by $(R'_+,R'_-)$ contains the derived subgroup of $F$, as the equivalence relation $\sim$ is not affected. 

(2) The proof is similar to the proof of (1).  
First, let $u,v$ and $w$ be as in part (1). By assumption, the element $f$ has non-trivial slope at $1^-$. Replacing $f$ by $f^{-1}$ if necessary, we can assume that $f'(1^-)>1$. Hence, there exist $m>\ell$ in $\mathbb{N}$ such that $f$ has the pair of branches $1^{m}\to 1^{m-\ell}$.

Let $T'$ be a finite binary tree such that $u,v0,v1,w0,w10,w11$ are branches of $T'$. Let $T$ be the tree obtained from $T'$ by attaching the minimal binary tree with branch $1^m$ to the right-most leaf of $T'$. 
Let $r$ be the length of the last branch of $T'$ and note that $T$ has the branches $1^r0,1^{r+1}0,\dots,1^{r+m-1}0,1^{r+m}$. (In particular, since $m\geq 2$, there are at least three branches in $T$ after the branch $w11$.)

 Let $(S_+,S_-)$, $T_1$, $T_2$, $C_1$ and $C_2$ be as in part (1). 
Let $T_3$ be the minimal binary tree
with branch $00$ and $T_4$ be a tree which consists of a single caret.

Let $u_1,\dots,u_n$ be the branches of $T$.	 We use two copies of the tree $T$ to construct a new tree-diagram $(R_+,R_-)$ in $F$, by performing the
operations $(1a)-(3a)$ on the first copy of $T$ and $(1b)-(5b)$ on the second copy of $T$, as in the proof of part $(1)$. Note that there exists $5\leq k\leq n-5$ such that the branch $u_k\equiv w0$. By construction, the tree-diagram $(R_+,R_-)$ has the sets $\mbox{(A)}$ and $\mbox{(B)}$ of pairs of branches as in the proof of part (1) as well as the following set 
\begin{align*}
	\mbox{(C')}& \begin{cases}
		w11\to w110& \\
		u_{k+3}	\to w111\\
		u_{i} \to u_{i-1} & \mbox{for } k+4\leq i\leq n-1\\
		u_n00\to u_{n-1}\\
		u_n01\to u_n0\\
		u_n1\to u_n 1\\
	\end{cases}
\end{align*}

Let $g\in F$ be the element represented by $(R_+,R_-)$. Then $\pi_{ab}(g)=(c,0)$. Let $H=\la f,g\ra$. The proof that $H$ contains the derived subgroup of $F$ is almost identical to the proof in part (1) (note that the first $3$ rows in $\mbox{(C')}$ also coincide with the first three rows in $\mbox{(C)}$). The only difference is in the proof that Condition (4) of Lemma \ref{lem:tree} holds for $H$ (with  the tree $T$ and the word $w$).
Let us prove that the condition holds. That is, we claim that for every $i\geq 0$, we have $u_n1^i0\sim w$ (where $\sim$ stands for $\sim_H$). 
Recall that by construction $u_n\equiv 1^{r+m}$. Hence, we need to prove that for every $i\geq 0$, we have $1^{r+m+i}0\sim w$. Clearly, it suffices to prove that for every $i\geq 0$, we have $1^{r+i}0\sim w$.  By Condition (1) of Lemma \ref{lem:tree} (which holds here, as in part (1)), for every $j\in\{2,\dots,n-1\}$ we have $u_j\sim w$. Since $1^r0,1^{r+1}0,\dots,1^{r+m-1}0$ are branches of $T$ (which are clearly not the first nor the last branch of $T$), for every $0\leq i\leq m-1$, we have $1^{r+i}0\sim w$. We prove by induction that the claim also holds for every $i\geq m$. Indeed, let $i\geq m$. Since $f$ has the pair of branches $1^m\to 1^{m-\ell}$, we have $1^m\sim 1^{m-\ell}$. It follows that $1^{r+i}0\sim 1^{r+(i-\ell)}0$. Since $i-\ell<i$, we are done by induction. Hence, Condition (4) of Lemma \ref{lem:tree} holds for $H$, as required.  
\end{proof}

Theorem \ref{thm:main intro} is a corollary of Proposition \ref{prop:main}. Recall the statement of the theorem.

\begin{Theorem2}\label{thm:main}
	Let $(a,b),(c,d)\in\mathbb{Z}^2$ be such that  $\{a,c\}\neq\{0\}$ and $\{b,d\}\neq\{0\}$. Let $f\in F$ be a non-trivial element such that $\pi_{ab}(f)=(a,b)$. Then there exists an element $g\in F$ such that $\pi_{ab}(g)=(c,d)$ and such that $\la f,g\ra=\pi_{ab}^{-1}(\la (a,b),(c,d)\ra)$.
\end{Theorem2}

\begin{proof}
	It suffices to prove that there exists an element $g\in F$ such that $\pi_{ab}(g)=(c,d)$ and such that $\la f,g\ra$ contains the derived subgroup of $F$. Indeed, in that case, $\pi_{ab}(\la f,g\ra)=\la (a,b),(c,d)\ra$ and since $\la f,g\ra$ contains the derived subgroup of $F$, we have that $$\la f,g\ra=\pi_{ab}^{-1}(\la(a,b),(c,d)\ra),$$ as necessary. 
	
	If $c,d\neq 0$, then by Proposition \ref{prop:main}(1), there is an element $g\in F$ such that $\pi_{ab}(g)=(c,d)$ and such that $\la f,g\ra$ contains the derived subgroup of $F$, as required.
	
	If $c\neq 0$ and $d=0$, then by assumption $b\neq 0$. Hence, $f$ has non-trivial slope at $1^-$. Then by Proposition \ref{prop:main}(2), there is an element $g\in F$ such that $\pi_{ab}(g)=(c,0)=(c,d)$ and such that $\la f,g\ra$ contains the derived subgroup of $F$, as necessary.
	
	If $c=0$ and $d\neq 0$, we are done in a similar way, using Proposition \ref{prop:main}(3). 
	
	Finally, if $c=d=0$, then $a\neq 0$ and $b\neq 0$. Then $f$ has non-trivial slope both at $0^+$ and at $1^-$. Hence, by Proposition \ref{prop:main}(4), there is an element $g\in F$ such that $\pi_{ab}(g)=(0,0)=(c,d)$ and such that $\la f,g\ra$ contains the derived subgroup of $F$, as necessary.
\end{proof}

Finally, Theorem \ref{thm:almost3/2} is a corollary of Theorem \ref{thm:main intro}.
 Indeed, let $f\in F$ be such that $\pi_{ab}(f)=(a,b)$ forms part of a generating pair of $\mathbb{Z}^2$. Then there exists $(c,d)\in\mathbb{Z}^2$ such that $\{(a,b),(c,d)\}$ is a generating set of $\mathbb{Z}^2$. Since $(a,b)$ and $(c,d)$ generate $\mathbb{Z}^2$, we have that $\{a,c\}\neq \{0\}$ and $\{b,d\}\neq \{0\}$. Hence, by Theorem \ref{thm:main intro}, there is an element $g\in F$ such that $\pi_{ab}(g)=(c,d)$ and such that 
$$	\la f,g\ra=\pi_{ab}^{-1}(\la (a,b),(c,d)\ra)=\pi_{ab}^{-1}(\mathbb{Z}^2)=F.$$
Hence, $f$ is part of a $2$-generating set of $F$. 
\qed

\begin{minipage}{3 in}
	Gili Golan\\
	Department of Mathematics,\\
	Ben Gurion University of the Negev,\\ 
	golangi@bgu.ac.il
\end{minipage}

\begin{thebibliography}{999999}
	
	



\bibitem{BBQS} J. Belk, C. Bleak, M. Quick and R. Skipper, 
\it Type systems and maximal subgroups of Thompson's group V,
\rm arxiv:2206.12631. 



\bibitem{BHS} C. Bleak, S. Harper and R. Skipper, 
\it Thompson's group $T$ is $\frac{3}{2}$-generated,
\rm arxiv:2206.05316

\bibitem{BW} C. Bleak and B. Wassink,
\it Finite index subgroups of R. Thompson’s group F. 
\rm arXiv:0711.1014 

\bibitem{BGK} T. Breuer, R. Guralnick and W.  Kantor,
\it Probabilistic generation of finite simple
groups, II,
\rm J. Algebra 320 (2008), 443–494.

\bibitem{Brin} M. Brin,
\it Higher dimensional Thompson groups,
\rm Geom. Dedicata 108 (2004), 163–192.

\bibitem{BCR} J. Burillo, S. Cleary and E. R$\mathrm{\ddot{o}}$ver, 
\it Commensurations and Subgroups of Finite Index of Thompson's Group $F$, 
\rm Geom. Topol. 12(3): 1701-1709 (2008).

\bibitem{BGH} T. Burness, R. Guralnick and S. Harper,
\it The spread of a finite group,
\rm Ann. of Math. (2) 193 (2021), 619–687.





\bibitem{CERT} S. Cleary, M. Elder, A. Rechnitzer, J. Taback,
\it Random subgroups of Thompson's group $F$,
\rm Groups Geom. Dyn. 4 (2010), no. 1, 91-126.






\bibitem{CFP} J. Cannon, W. Floyd, and W. Parry,
\it  Introductory notes on Richard Thompson's groups.
\rm L'Enseignement Mathematique, 42 (1996), 215--256.

\bibitem{Cox} C. Cox, 
\it  On the spread of infinite groups.
\rm Proceedings of the Edinburgh Mathematical Society. 2022 Feb;65(1):214-28.




\bibitem{DH} C. Donoven and S. Harper, 
\it Infinite $\frac{3}{2}$-generated groups,
\rm Bull. Lond. Math. Soc. 52 (2020), 657–673.








\bibitem{GGJ} T. Gelander, G. Golan and K. Juschenko,
\it Invariable generation of Thompson groups,
\rm Journal of Algebra 478 (2017), 261--270.


\bibitem{GS} G. Golan and M. Sapir,
\it On subgroups of R. Thompson group $F$,
\rm Trans. Amer. Math. Soc. 369 (2017), 8857--8878.


\bibitem{G16} G. Golan,
\it The generation problem in Thompson group $F$,
\rm arxiv:1608.02572, to appear in Memoirs of the AMS.

\bibitem{G-RG} G. Golan Polak, 
\it Random Generation of Thompson's group $F$, 
\rm Journal of Algebra, 593 (2022), 507-524.

\bibitem{G22} G. Golan Polak,
\it On Maximal subgroups of Thompson's group $F$, 
\rm arxiv:2209.03244. 

\bibitem{Guba} V. Guba, 
\it A finite generated simple group with free 2-generated subgroups,
\rm Sibirsk. Mat.Zh. 27 (1986), 50–67.

\bibitem{GK} R. Guralnick and W. Kantor, \it Probabilistic generation of finite simple groups, 
\rm J. Algebra 234
(2000), 743–792.

\bibitem{Higman} G. Higman,
\it Finitely presented infinite simple groups,
\rm Notes on Pure Mathematics, vol. 8,
Department of Mathematics, I.A.S., Australia National University, Canberra, 1974.










\bibitem{O} A. Olshanskii,
\bfseries{Geometry of defining relations in groups},
\rm (Nauka, Moscow, 1989) (in Russian); (English translation by Kluwer Publications 1991).







\bibitem{S} R. Steinberg, 
\it Generators for simple groups,
\rm Canadian J. Math. 14 (1962), 277–283.





%





\end{thebibliography}
\end{document}